\documentclass[12pt]{article}
\usepackage{amssymb,amsmath,amsthm,tikz,multirow,amscd}
\usetikzlibrary{arrows,calc}

\title{Fixed Point Sets and the Fundamental Group I: Semi-free Actions on $G$-CW-Complexes}
\author{Sylvain Cappell \\
Courant Institute of Mathematical Sciences, New York University \\
\\
Shmuel Weinberger\thanks{Research was partially supported by NSF grants DMS 0504721 and 0805913} \\
University of Chicago \\
\\
Min Yan\thanks{Research was supported by Hong Kong RGC General Research Fund 16319116.} \\ Hong Kong University of Science and Technology}

\usepackage[hidelinks, hyperindex]{hyperref}

\hyperref[sec:function]{}

\newcommand{\sub}{\subset}

\newcommand{\mc}{\mathcal}
\newcommand{\bb}{\mathbb}

\newtheorem{theorem}{Theorem}
\newtheorem*{theorem*}{Theorem}
\newtheorem*{mtheorem*}{Main Theorem}
\newtheorem{lemma}[theorem]{Lemma}

\newtheorem{proposition}[theorem]{Proposition}

\theoremstyle{definition}
\newtheorem*{definition*}{Definition}
\newtheorem*{case*}{Case}
\newtheorem*{subcase*}{Subcase}

\newtheorem{example}{Example}

\numberwithin{equation}{section}

\begin{document}

\maketitle

\begin{abstract}
Smith theory says that the fixed point of a semi-free action of a group $G$ on a contractible space is ${\bb Z}_p$-acyclic for any prime factor $p$ of $G$. Jones proved the converse of Smith theory for the case $G$ is a cyclic group acting on finite CW-complexes. We extend the theory to semi-free group action on finite CW-complexes of given homotopy type, in various settings. In particular, the converse of Smith theory holds if and only if certain $K$-theoretical obstruction vanishes. We also give some examples that show the effects of different types of the $K$-theoretical obstruction.
\end{abstract}

\section{Introduction}

The homological theory of group actions began with the results of P.A. Smith \cite{smith} that, if $G$ is a $p$-group acting on a contractible space, then the fixed set is ${\bb Z}_p$-acyclic. While originally the connection between the order of the group and the nature of acyclicity seemed like an artifact of the proof, it was quickly realized that this was not the case.

The definitive refutation of this was the result of L. Jones that any ${\bb Z}_n$-acyclic finite complex is the fixed set of semi-free ${\bb Z}_n$-action on a finite contractible complex \cite{jones}. Here we recall that a group action is semi-free if all isotropy groups are either trivial or the whole group. If one removes the semi-free condition, then R. Oliver's work \cite{oliver} shows that the necessary and sufficient condition for $F$ to be a fixed set is that the Euler characteristic $\chi(F)=1$. Incidentally, this is not necessary for general topological actions, but it is for the so called ANR-actions, which is somewhat weaker than having the equivariant homotopy type of a finite $G$-complex.

This paper and a companion one \cite{cwy-oliver} study the extensions of the work of Jones and Oliver to nonsimply connected spaces. The simply connected theory was interestingly explored by Assadi \cite{assadi} and Oliver-Petrie \cite{op}, and is largely understood. Both theories depend on a kind of ``equivariant surgery'' and involve $K_0$. Assadi-Vogel \cite{av} developed a non-simply connected semi-free theory for actions on manifolds (therefore only for certain restricted family of groups). Our work extends theirs, in the situation of finite complexes acted upon by many more finite groups.

In this paper, we will see that, even for $G={\bb Z}_p$ there are a rich set of phenomena visible in trying to understand the homotopy types of fixed sets, in contrast to the situation for non-$p$-groups. The main results of \cite{cwy-oliver} show that the Euler characteristic conditions do not become substantially more subtle in the presence of the fundamental group.

\begin{theorem*}
For a finite complex $X$ and a cyclic group $G={\bb Z}_n$, $n$ not prime power, the following are equivalent:
\begin{enumerate}
\item There is a fixed point free finite $G$-complex $Y$ with a $G$-map $Y\to X$ which is a homotopy equivalence.
\item $\chi(X)=0$.
\end{enumerate}
\end{theorem*}

The necessity of the condition is a consequence of the Lefschetz fixed point theorem, and therefore also holds for the $G$-ANR case. Indeed, the theorem is the reason the condition $\chi(F)=1$ arises in Oliver's classification of possible fixed sets of ${\bb Z}_n$-actions on contractible complexes. The sufficiency in \cite{cwy-oliver} builds on Oliver's work by a series of purely geometric constructions; for our purposes, we remark that the fundamental group of $X$ does not enter.

In contrast to the generalization of Oliver's theorem, an analysis of semi-free actions shows a number of interesting phenomena. We will mention some examples before describing the theorems.

\begin{example}\label{example1}
Let $T(r)$ be the mapping torus of a degree $r$ map from a sphere $S^d$ to itself. Notice that the map $T(r)\to S^1$ is a ${\bb Z}_n$-homology equivalence if and only if $n$ divides a power of $r$ (i.e., all the primes in $n$ occur in $r$). The infinite cyclic cover has nontrivial ${\bb Q}$-homology, but is also  ${\bb Z}_n$-acyclic under this divisibility condition.

We will see that there is a semi-free ${\bb Z}_n$-action on a finite complex homotopy equivalent to $S^1$ with fixed set $T(r)$ iff $n|r$. When $n$ is not square-free, this condition goes beyond Smith theory. It is related to $\tilde{K}_0({\bb Z}[{\bb Z}\times{\bb Z}_n])$, and the role of square-free is well known to be the condition for Nil to be nontrivial in the Bass-Heller-Swan formula (see Bass-Murthy \cite{bm}). Concretely, $T(p)$ is fixed under ${\bb Z}_p$-action, but not a semi-free ${\bb Z}_{p^2}$-action. Its two fold cover $T(p^2)$ is fixed under a semi-free ${\bb Z}_{p^2}$-action, but not a semi-free ${\bb Z}_{p^3}$-action, etc.
\end{example}

If one studies topological actions that are locally smooth, one does not necessarily obtain a finite $G$-complex \cite{quinn1,steinberger,we1}. The non-uniqueness of such structures, even when they exist, is implicated in the phenomenon of nonlinear similarity of linear actions on the sphere \cite{cs}. In the above examples one can obtain a locally smooth action (or equivalently a $G$-ANR action) with $T(r)$ as fixed set if and only if one can construct a $G$-action on finite complex. The following example shows a difference between these categories.

\begin{example}
Let $T(r_1,r_2)$ be the double mapping torus, obtained by glueing two ends of $S^d\times [0,1]$ to a copy of $S^d$ by maps of degrees $r_1,r_2$. Then $T(2,3)$ is ${\bb Z}_6[{\bb Z}]$-acylic. It is the fixed set of a semi-free ${\bb Z}_6$-ANR action. On the other hand, it is not the fixed set (up to homotopy) of any finite $G$-complex homotopy equivalent to the circle $S^1$.

In this case, the obstructions are nontrivial elements of $K_{-1}({\bb Z}[{\bb Z}_6])$ that enter via the Bass-Heller-Swan formula into the obstruction group $\tilde{K}_0({\bb Z}[{\bb Z}\times {\bb Z}_6])$; there are similar examples arising for all non-prime-power groups. 
\end{example}

One of the reasons to focus so strongly on the case of the circle is the special role that it plays in the Farrell-Jones conjecture \cite{fj}. The circle is central to this problem because, as we shall soon explain more systematically, the examples on the circle can be promoted to examples on any finite complex whose fundamental group is a torsion free hyperbolic group, or a lattice.

Assuming the Farell-Jones conjecture, if the fundamental group is torsion free, there would not be any example of fixed sets obstructed for ${\bb Z}_p$-actions when the Smith condition holds. To give an example where there is an obstruction, we turn to finite fundamental group.

\begin{example}
Let $f\colon L_3(kp) \to L_3(p)$ be a degree $r$ map of three dimensional lens spaces, with $k,p$ coprime. There is a ${\bb Z}_p$-action on a space of the homotopy type of $L_3(p)$, such that the inclusion map from the fixed set is homotopic to $f$, if and only if $r^{p-1}=k^{p-1}$ mod $p^2$. The details are in Proposition \ref{lens}.
\end{example}

We now state the results from which the above examples follow.

Let $G$ be a group. A $G$-map between finite $G$-complexes (or compact $G$-ANRs) is a {\em pseudo-equivalence} if it is a homotopy equivalence without considering the group action. Given a $G$-map $f\colon F\to Y$, we ask whether it is possible to extend $F$ to a bigger $G$-space $X$, and extend $f$ to a pseudo-equivalent $G$-map $g\colon X\to Y$. We call $g$ a {\em pseudo-equivalence extension} of $f$.

In this paper, we concentrate in the following setting. The group $G$ is finite, and all spaces are finite semi-free $G$-complexes. Moreover, we only consider $F=X^G$ in the pseudo-equivalence extension. In other words, the extension from $F$ to $X$ is obtained by attaching free $G$-cells.

For the special case $Y$ is a point, the question of the existence of a pseudo-equivalence extension becomes whether a given space $F$ can be the fixed point set of a semi-free $G$-action on a contractible space $X$. The classical results of Smith \cite{smith} and Jones \cite{jones} give necessary and sufficient condition for semi-free actions by cyclic groups.

\begin{theorem*}[Smith and Jones]
A finite complex $F$ is the fixed set of a finite contractible semi-free ${\bb Z}_n$-complex if and only if $\tilde{H}_*(F;{\bb Z}_n)=0$.
\end{theorem*}

For a general semi-free action of $G$ on contractible $X$, and any prime factor $p$ of $|G|$, the fixed set $F=X^G$ is the same as the fixed set $X^C$ of a cyclic subgroup of order $p$. Then the homological condition in the theorem becomes $\tilde{H}_*(F;{\bb F}_p)=0$ for all prime factors $p$ of $G$ (${\bb F}_p={\bb Z}_p$ is a field for prime $p$). This is equivalent to $\tilde{H}_*(F;{\bb Z}_{|G|})=0$. We call this the {\em Smith condition}.

In general, we let $Y$ be a connected semi-free $G$-complex. Let $\tilde{Y}$ be the universal cover of $Y$, with action by the fundamental group $\pi=\pi_1Y$. Then all actions on $\tilde{Y}$ covering $G$-actions on $Y$ form a group $\Gamma$ that fits into an exact sequence
\[
1\to \pi\to \Gamma\to G\to 1.
\]
In particular, if $G$ acts trivially on $Y$, then $\Gamma=\pi\times G$. 

Suppose a $G$-map $g\colon X\to Y$ is a pseudo-equivalence between semi-free $G$-complexes. Then the mapping cone of $g$ is a contractible semi-free $G$-complex, and Smith condition can be applied to the mapping cone to give isomorphisms $\tilde{H}_*(X^G;{\bb F}_p)\cong \tilde{H}_*(Y^G;{\bb F}_p)$ for all prime factors $p$ of $|G|$. In fact, in Section \ref{scondition}, we apply the Smith condition to the universal cover and get isomorphisms $\tilde{H}_*(X^G;{\bb F}_p\pi)\cong \tilde{H}_*(Y^G;{\bb F}_p\pi)$. This is the necessary Smith condition for constructing pseudo-equivalence extension. 

However, it turns out that there is also an algebraic $K$-theoretic obstruction. The following is our first main result, for the case the $G$-action on $Y$ is trivial. 

\begin{theorem}\label{trivialtarget}
Suppose $f\colon F\to Y$ is a map of finite complexes, with $Y$ connected and $\pi=\pi_1Y$. Then $F$ can be the fixed set of a finite semi-free $G$-complex $X$, and $f$ has pseudo-equivalence extension $g\colon X\to Y$, if and only if the following are satisfied.
\begin{enumerate}
\item The map $f$ induces isomorphisms $\tilde{H}_*(F;{\bb F}_p\pi)\cong \tilde{H}_*(Y;{\bb F}_p\pi)$ for all prime factors $p$ of $|G|$. 
\item An obstruction $[C_*(\tilde{f})]\in \tilde{K}_0({\bb Z}[\pi\times G])$ vanishes.
\end{enumerate}
\end{theorem}

The first is the Smith condition. In the second condition, the chain complex $C_*(\tilde{f})$ of the $\pi$-cover $\tilde{f}\colon \tilde{F}\to \tilde{Y}$ of $f$ is a ${\bb Z}\pi$-chain complex. Then we regard $C_*(\tilde{f})$ as a ${\bb Z}[\pi\times G]$-chain complex with trivial $G$-action. We will argue that the Smith condition implies that $C_*(\tilde{f})$ has a finite resolution of finitely generated ${\bb Z}[\pi\times G]$-projective modules, and therefore gives a well-defined element $[C_*(\tilde{f})]\in \tilde{K}_0({\bb Z}[\pi\times G])$. Moreover, since the terms in $C_*(\tilde{f})$ are finitely generated free ${\bb Z}[\pi]$-modules, the obstruction lies in the kernel of the homomorphism that forgets the $G$-action 
\[
[C_*(\tilde{f})]\in
\text{Ker}(\;\tilde{K}_0({\bb Z}[\pi\times G])\to \tilde{K}_0({\bb Z}[\pi])\;).
\]

\begin{theorem}\label{nontrivialtarget}
Suppose $Y$ is a finite semi-free connected $G$-complex, with $\pi=\pi_1Y$. Suppose $F$ is a finite complex and $f\colon F\to Y^G$ is a map. Then $F$ can be the fixed set of a finite semi-free $G$-complex $X$, and $f$ has pseudo-equivalence extension $g\colon X\to Y$, if and only if the following are satisfied.
\begin{enumerate}
\item The map $f$ induces isomorphisms $H_*(F;{\bb F}_p\pi)\cong H_*(Y^G;{\bb F}_p\pi)$ for all prime factors $p$ of $|G|$. 
\item An obstruction $[C_*(\tilde{f})]\in \tilde{K}_0({\bb Z}[\Gamma])$ vanishes.
\end{enumerate}
\end{theorem}

The meaning of the two conditions is explained in Sections \ref{scondition} and \ref{proof}. The Smith condition is equivalent to that the condition is satisfied on each connected component of $Y^G$. Then we get a $K$-theory element on each component similar to the first main theorem, and the obstruction $[C_*(\tilde{f})]$ is the sum of these. Moreover, similar to the remark for Theorem \ref{trivialtarget}, we know the obstruction lies in the kernel of the forgetful homomorphism $\tilde{K}_0({\bb Z}[\Gamma])\to \tilde{K}_0({\bb Z}[\pi])$.

We also describe how $\tilde{K}^{top}$ enters to modify the above results in the simple homotopy setting (Theorem \ref{jones_simple}) and $G$-ANR setting (Theorem \ref{jones_anr}).

We remark that Oliver and Petrie \cite{op} studied the extension problem in general different setting (see also Assadi \cite{assadi}, and Morimoto and Iizuka \cite{mi}). When restricted to our problem, they gave the obstruction such that the extension $g$ induces isomorphism on the integral homology. Therefore they solved our problem for the case $Y$ is simply connected. What is new in our theorem is the non-simply connected case for homotopy equivalences.

Another important paper in this direction was Assadi-Vogel \cite{av}, that works in a manifold setting. It is quite close to what we do, although their techniques are different (based on ideas of homology propagation rather than $G$-surgery), formally have less generality (since ${\bb Z}_p\times{\bb Z}_p$ cannot act semi-freely on a manifold, for example) and their calculations focus on finite fundamental groups. Our focus here is mainly on the phenomena that arise when fundamental groups are torsion free, as this paper is intended to provide foundations for later studies of group actions on aspherical manifolds.

\section{Smith Condition}
\label{scondition}

We now explain the Smith condition in more detail. Let $G$ act on connected $Y$, and let $p\colon \tilde{Y}\to Y$ be the universal cover, with the free action by the fundamental group $\pi=\pi_1Y$. The actions lift to self homeomorphisms of the universal cover, and form a group $\Gamma$ fitting into an exact sequence 
\begin{equation}\label{lift}
1\to \pi\to \Gamma\to G\to 1.
\end{equation}

As an example, consider $G={\bb Z}_2=\langle g\rangle$ acting on the real projective space ${\bb RP}^2$ by $g([x_0,x_1,x_2])=[x_0,x_1,-x_2]=[-x_0,-x_1,x_2]$. The universal cover of ${\bb RP}^2$ is the sphere $S^2$, with the covering group $\pi$ generated by the antipode $a(x_1,x_2,x_3)=(-x_1,-x_2,-x_3)$. The action $g$ lifts to $\tilde{g}_1(x_0,x_1,x_2)=(x_0,x_1,-x_2)$ and $\tilde{g}_2(x_0,x_1,x_2)=(-x_0,-x_1,x_2)$. The group $\Gamma={\bb Z}_2\times {\bb Z}_2=\langle \tilde{g}_1\rangle\times \langle \tilde{g}_2\rangle$, and $\pi$ is a subgroup of $\Gamma$ by $a=\tilde{g}_1\tilde{g}_2$.

We use $\tilde{?}$ to denote the lifting/pullback of $?$ along the universal cover of $Y$. For example, we have the pullbacks
\[
\begin{CD}
\tilde{X}  @>\tilde{g}>> 
\widetilde{Y} \\
@VVV  @VVpV \\
X @>g>>
Y
\end{CD}
\qquad\qquad
\begin{CD}
\tilde{F}  @>\tilde{f}>> 
\widetilde{Y^G}=p^{-1}(Y^G) \\
@VVV  @VVpV \\
F=X^G @>f>>
Y^G
\end{CD}
\]
For a connected component $C$ of $Y^G$, we have the pullback
\[
\begin{CD}
\tilde{F}_C=\tilde{f}^{-1}(\tilde{C})  @>>> 
\tilde{C}=p^{-1}(C) \\
@VVV  @VVV \\
F_C=f^{-1}(C) @>>>
C
\end{CD}
\]
In our example, we have $({\bb RP}^2)^G=\{[x_0,x_1,0]\}\sqcup\{[0,0,1]\}={\bb RP}^1\sqcup \{[0,0,1]\}$, and $\widetilde{({\bb RP}^2)^G}=\{(x_0,x_1,0)\}\sqcup\{(0,0,1)\}\sqcup\{(0,0,-1)\}=S^1\sqcup N\sqcup S$ ($N$ and $S$ are the north and south poles).

Let $\tilde{y}\in \tilde{Y}$ and $y=p(\tilde{y})\in Y$. Let $C$ and $\hat{C}$ be the components of $Y^G$ and $\widetilde{Y^G}$ containing $y$ and $\tilde{y}$. Then $\hat{C}$ covers $C$. We may use $\tilde{y}$ to get an isomorphism $\pi_1(Y,y)\cong \pi$. Then the deck transformations $\pi_{\hat{C}}\sub \pi$ of the covering $\hat{C}\to C$ is the image of the homomorphism $\pi_1(C,y)\to \pi_1(Y,y)$, and we get $\tilde{C}=\pi\times_{\pi_{\hat{C}}} \hat{C}$. 

The induced homomorphism $\Gamma_{\tilde{y}}\to G_y$ of isotropy groups is an isomorphism. In particular, for $y\in Y^G$, the isomorphism gives a splitting $G=G_y\cong \Gamma_{\tilde{y}}\sub \Gamma$ of \eqref{lift}. The splitting depends only on the component $\hat{C}$. Therefore we may denote $\Gamma_{\hat{C}}=\Gamma_{\tilde{y}}$, and get $\Gamma=\pi\rtimes\Gamma_{\hat{C}}$. 

The other connected components of $\tilde{C}$ are $a\hat{C}$, $a\in \pi$. Therefore the component $C$ gives a $\pi$-conjugation class of isotropy groups (equivalently, a $\pi$-conjugation class of splittings of \eqref{lift})
\[
\Gamma_C=\{\Gamma_{a\hat{C}}=a\Gamma_{\hat{C}}a^{-1}\colon a\in \pi\}.
\]

In our example, $({\bb RP}^2)^G$ has two connected components $C_1={\bb RP}^1$ and $C_2=[0,0,1]$. Their preimages in $\widetilde{({\bb RP}^2)^G}$ are respectively $\tilde{C}_1=S^1$ and $\tilde{C}_2=\{N,S\}$. We may take $\hat{C}_1=S^1$, $\hat{C}_2=N$, with $a\hat{C}_1=\hat{C}_1$, $a\hat{C}_2=S$. Then $\Gamma_{S^1}=\langle \tilde{g}_1\rangle$, $\Gamma_N=\Gamma_S=\langle \tilde{g}_2\rangle$. The semi-direct products $\Gamma=\pi\rtimes\Gamma_{S^1}=\langle a\rangle\times \langle \tilde{g}_1\rangle$ and $\Gamma=\pi\rtimes\Gamma_N=\pi\rtimes\Gamma_S=\langle a\rangle\times \langle \tilde{g}_2\rangle$ are the usual products. Moreover, we have two fold cover $\hat{C}_1=S^1\to C_1={\bb RP}^1$ with covering group $\pi_{\hat{C}_1}=\pi$, and $\tilde{C}_1=\hat{C}_1=\pi\times_{\pi}\hat{C}_1$. We also have one fold cover $\hat{C}_2\to C_2$ with covering group $\pi_{\hat{C}_2}=1$, and $\tilde{C}_2=\hat{C}_2\sqcup a\hat{C}_2=\pi\times_1 \hat{C}_2$. 

Denote the homomorphism $G\cong \Gamma_{\hat{C}}\sub \Gamma$ by $u\to\tilde{u}$. Then the elements of $\Gamma=\pi\rtimes \Gamma_{\hat{C}}\cong \pi\rtimes G$ are $a\tilde{u}$, with $a\in \pi$ and $u\in G$. The multiplication in $\Gamma$ is given by $a_1\tilde{u}_1a_2\tilde{u}_2=a_1u_1(a_2)\widetilde{u_1u_2}$. Here $u(a)$ is obtained by regarding $a\in \pi_1(Y,y)$ as a loop at $y$ and applying the action of $u\in G$ to the loop. In particular, if $a\in \pi_1(C,y)$ lies in the deck transformation group $\pi_{{\hat{C}}}$, then $u(a)=a$. Therefore $\pi_{\hat{C}}\times\Gamma_{\hat{C}}\cong \pi_{\hat{C}}\times G$ is a subgroup of $\Gamma$.

We have ${\bb Z}[\Gamma]$-chain complexes
\[
C_*(\widetilde{Y^G})
=\oplus_{C\in \pi_0Y^G} C_*(\tilde{C}),\quad
C_*(\tilde{C})
=C_*(\pi\times_{\pi_{\hat{C}}} {\hat{C}})
={\bb Z}\pi\otimes_{{\bb Z}\pi_{\hat{C}}}C_*({\hat{C}}).
\]
Since the isotropy group $\Gamma_{\hat{C}}$ acts trivially on ${\hat{C}}$, we may regard $C_*({\hat{C}})$ as a ${\bb Z}[\pi_{\hat{C}}\times \Gamma_{\hat{C}}]$-module (with trivial action by $\Gamma_{\hat{C}}$). By $\Gamma=\pi\rtimes \Gamma_{\hat{C}}$, we get the following interpretation of the ${\bb Z}[\Gamma]$-chain complex $C_*(\tilde{C})$
\[
C_*(\tilde{C})
={\bb Z}[\pi\rtimes \Gamma_{\hat{C}}]\otimes_{{\bb Z}[\pi_{\hat{C}}\times \Gamma_{\hat{C}}]}C_*({\hat{C}})
=\text{Ind}_{{\bb Z}[\pi_{\hat{C}}\times G]}^{{\bb Z}[\pi\rtimes G]}C_*({\hat{C}}).
\]
Similarly, let $\hat{F}_C\sub \tilde{F}_C$ correspond to $\hat{C}\sub \tilde{C}$. Then we have 
\[
C_*(\widetilde{F})
=\oplus_{C\in \pi_0Y^G}C_*(\widetilde{F}_C),\quad
C_*(\widetilde{F}_C)
=\text{Ind}_{{\bb Z}[\pi_Z\times G]}^{{\bb Z}[\pi\rtimes G]}C_*(\hat{F}_C).
\]

The pseudo-equivalent $G$-map $g\colon X\to Y$ between semi-free $G$-complexes lifts to a $\Gamma$-map $\tilde{g}\colon \tilde{X}\to \tilde{Y}$, and $\tilde{g}$ has the ``fixed part'' $\tilde{f}\colon\tilde{F}\to \widetilde{Y^G}$. Here the fixed part is not fixed by the whole $\Gamma$, but by various isotropy subgroups $\Gamma_{\hat{C}}\sub \Gamma$. We regard the pseudo-equivalent $\Gamma$-map $\tilde{g}\colon \tilde{X}\to \tilde{Y}$ as a pseudo-equivalent $\Gamma_{\hat{C}}$-map. This implies that the mapping cone $C(\tilde{g})$ is a contractible, semi-free $\Gamma_{\hat{C}}$-space. The fixed set $C(\tilde{g})^{\Gamma_{\hat{C}}}$ is the mapping cone of the map $\tilde{X}^{\Gamma_{\hat{C}}}\to \tilde{Y}^{\Gamma_{\hat{C}}}$. By the classical Smith theory \cite{smith}, $C(\tilde{g})^{\Gamma_{\hat{C}}}$ has trivial ${\bb F}_p$-homology for all prime factors $p$ of $|\Gamma_{\hat{C}}|=|G|$. This implies an isomorphism $\tilde{H}_*(\tilde{X}^{\Gamma_{\hat{C}}};{\bb F}_p)\cong \tilde{H}_*(\tilde{Y}^{\Gamma_{\hat{C}}};{\bb F}_p)$.

We note that $\Gamma_{\hat{C}}$ may fix several components $\hat{C}_0,\hat{C}_1,\dots,\hat{C}_k$ of $\widetilde{Y^G}$, in addition to $\hat{C}_0=\hat{C}$. Then the isomorphism $\tilde{H}_*(\tilde{X}^{\Gamma_{\hat{C}}};{\bb F}_p)\cong \tilde{H}_*(\tilde{Y}^{\Gamma_{\hat{C}}};{\bb F}_p)$ is a direct sum of isomorphisms $\tilde{H}_*(\hat{F}_{C_i};{\bb F}_p)
\cong
\tilde{H}_*(\hat{C}_i;{\bb F}_p)$. For $i=0$, this gives the {\em local Smith condition}
\begin{equation}\label{localsmith}
\tilde{H}_*(\hat{F}_C;{\bb F}_p)
\cong
\tilde{H}_*(\hat{C};{\bb F}_p)\quad
\text{for all prime factors $p$ of $|G|$}.
\end{equation}

For our example, we have the condition on $\hat{C}_1=S^1$ with the antipode action by $\pi=\langle a\rangle$, and the condition on $\hat{C}_2=N$ with the trivial group action. The first condition is obtained by applying the usual Smith condition to the action of  $\Gamma_{S^1}=\langle \tilde{g}_1\rangle$ on $S^2$. The second condition is obtained by applying the usual Smith condition to the action of $\Gamma_N=\Gamma_S=\langle \tilde{g}_2\rangle$ on $S^2$. Although the second condition consists of conditions on $N$ and $S$, the two conditions are equivalent by the action of $\pi=\langle a\rangle$.

Finally, we combine the local Smith conditions into the global Smith condition in Theorems \ref{trivialtarget} and \ref{nontrivialtarget}. We pick one connected component $\hat{C}$ of $\widetilde{C}$ for each component $C$ of $Y^G$. Then
\[
\widetilde{Y^G}
=\sqcup_{C\in \pi_0Y^G}\pi\times_{\pi_{\hat{C}}}\hat{C},
\]
and
\begin{align*}
\tilde{H}_*(Y^G;{\bb F}_p\pi)
&=\tilde{H}_*(\widetilde{Y^G};{\bb F}_p) \\
&=\oplus_{C\in \pi_0Y^G}\tilde{H}_*(\pi\times_{\pi_{\hat{C}}}\hat{C};{\bb F}_p)  \\
&=\oplus_{C\in \pi_0Y^G}{\bb F}_p\pi\otimes_{{\bb F}_p\pi_{\hat{C}}}\tilde{H}_*(\hat{C};{\bb F}_p),
\end{align*}
Similarly, we have
\[
\tilde{H}_*(F;{\bb F}_p\pi)
=\tilde{H}_*(\tilde{F};{\bb F}_p)
=\oplus_{C\in \pi_0Y^G}{\bb F}_p\pi\otimes_{{\bb F}_p\pi_{\hat{C}}}\tilde{H}_*(\hat{F}_C;{\bb F}_p).
\]
Then the direct sum of local Smith condition \eqref{localsmith} is the {\em global Smith condition}
\begin{equation}\label{globalsmith}
\tilde{H}_*(F;{\bb F}_p\pi)
\cong 
\tilde{H}_*(Y^G;{\bb F}_p\pi)\quad
\text{for all prime factors $p$ of $|G|$}.
\end{equation}
For our example, the global Smith condition is the direct sum of the Smith conditions on $S^1$, $N$, and $S$.

\section{Proof of Main Theorems}
\label{proof}

The proof of the main theorems in the introduction is an equivariant version of Wall's finiteness obstruction \cite{wall}.

\begin{proof}[Proof of Theorem \ref{trivialtarget}]
We assume that the first (Smith) condition is satisfied and try to construct a pseudo-equivalence extension $g$. In the process, we will encounter the obstruction in the second condition, and will see that it is well-defined.

By $\tilde{H}_0(F;{\bb F}_p\pi)\cong \tilde{H}_0(Y;{\bb F}_p\pi)=0$, we know $F$ is connected. We choose a base point in $F$ and use its image in $Y$ as the base point of $Y$. For any loop $\epsilon$ in $Y$ at the base point, we may attach $G$ copies of loops to $F$, and equivariantly map these loops to the loop $\epsilon$. Since $Y$ is a finite complex, we may attach finitely many such loops to get a semi-free $G$-complex $X^1$ with fixed point set $F$ and a $G$-map $f^1\colon X^1\to Y$ that is surjective on $\pi_1$. 

Since $X^1$ is a finite $G$-complex, there are finitely many loops $\epsilon_i$ generating $\pi_1X^1$. Since $f^1$ is surjective on $\pi_1$, the images $f^1(\epsilon_i)$ generate $\pi$, which is finitely presented because $Y$ is a finite complex. Therefore $\pi$ can be presented by $f^1(\epsilon_i)$ as generators, with finitely many words $f^1(w_j)$ of these loops as relations. For each such word $f^1(w_j)$, we glue $G$-copies of $D^2$ to $X^1$ along $Gw_j$ and equivariantly map these $2$-cells to $Y$. We thus get a semi-free $G$-complex $X^{1.5}$ with fixed point set $F$ and a $G$-map $f^{1.5}\colon X^{1.5}\to Y$ that induces an isomorphism on $\pi_1$. 

Since $f^{1.5}$ is isomorphic on $\pi_1$, by the Hurewicz theorem, we have $\pi_2(f^{1.5})=H_2(f^{1.5};{\bb Z}\pi)$. This implies that $\pi_2(f^{1.5})$ is finitely generated as a ${\bb Z}\pi$-module. In fact, as $G$ is a finite group, the $G$-action also makes $\pi_2(f^{1.5})$ into a finitely generated ${\bb Z}[\pi\times G]$-module. We represent a finite set of ${\bb Z}[\pi\times G]$-generators by maps $S^1\to X^{1.5}$ and $D^2\to Y$ compatible with $f^2$. Then we glue $G$-copies of $D^2$ along $G(S^1\to X^{1.5})$ to $X^{1.5}$ and equivariantly map these $2$-cells to $Y$ by $G(D^2\to Y)$ (for the current case that the $G$-action on $Y$ is trivial, this is $D^2\to Y$). We get a semi-free $G$-complex $X^2$ with fixed point set $F$, and extend $f^{1.5}$ to a $G$-map $f^2\colon X^2\to Y$, such that $f^2$ satisfies $\pi_1(f^2)=\pi_2(f^2)=0$. 

The construction from $f^{1.5}$ to $f^2$ can be inductively extended to higher dimensions. If we have a $G$-map $f^i\colon X^i\to Y$ satisfying $\pi_j(f^i)=0$ for $j\le i$, then we can use a finite set of ${\bb Z}[\pi\times G]$-generators of $\pi_{i+1}(f^i)=H_{i+1}(f^i;{\bb Z}\pi)$ to equivariantly attach $(i+1)$-dimensional free $G$-cells to $X^i$, and extend $f^i$ to a $G$-map $f^{i+1}\colon X^{i+1}\to Y$ satisfying $\pi_j(f^{i+1})=0$ for $j\le i+1$. Inductively, we get $f^n\colon X^n\to Y$ for some $n>\max\{\dim F,\dim Y\}$, such that $\pi_j(f^n)=0$ for $j\le n$. 

Let us consider the effect of one more construction to get $f^{n+1}\colon X^{n+1}\to Y$. The generators used for the construction can be interpreted as a basis of a finitely generated free ${\bb Z}[\pi\times G]$-module $A$ in a surjective ${\bb Z}[\pi\times G]$-homomorphism $A\to H_{n+1}(f^n;{\bb Z}\pi)$. By $n+1>\max\{\dim X^n,\dim Y\}$, we have $H_{n+2}(f^n;{\bb Z}\pi)=0$ and an exact sequence 
\begin{align*}
& H_{n+2}(f^n;{\bb Z}\pi)=0 \to
H_{n+2}(f^{n+1};{\bb Z}\pi) \to 
H_{n+1}(X^{n+1},X^n;{\bb Z}\pi)=A \\
& \to
H_{n+1}(f^n;{\bb Z}\pi)\to
H_{n+1}(f^{n+1};{\bb Z}\pi) \to
H_n(X^{n+1},X^n;{\bb Z}\pi)=0
\to\cdots
\end{align*}
The exact sequence gives $\pi_j(f^{n+1})=H_j(f^{n+1};{\bb Z}\pi)=0$ for $j\le n+1$, and then (by the Hurewicz theorem) a short exact sequence
\[
0 \to
\pi_{n+2}(f^{n+1})
=H_{n+2}(f^{n+1};{\bb Z}\pi) \to
A \to
H_{n+1}(f^n;{\bb Z}\pi)\to 0.
\]
Note that $\pi_{n+2}(f^{n+1})$ is to be used for the further construction based on $f^{n+1}$. The short exact sequence shows that, if $H_{n+1}(f^n;{\bb Z}\pi)$ has a finite resolution of finitely generated, free ${\bb Z}[\pi\times G]$-modules,
\[
0\to A_k\to \cdots\to A_2\to A_1\to H_{n+1}(f^n;{\bb Z}\pi)\to 0,
\]
then the resolution can be used as a recipe for constructing a $G$-map $f^{n+k}\colon X^{n+k}\to Y$, such that $G$ acts semi-freely on $X^{n+k}$ with $F$ as fixed point set, $f^{n+k}$ extends $f$ and is a non-equivariant homotopy equivalence. This $f^{n+k}$ is the pseudo-equivalence extension in the theorem.

Next we argue that the Smith condition implies that the ${\bb Z}[\pi\times G]$-module $H_{n+1}(f^n;{\bb Z}\pi)$ has a finite resolution by finitely generated, projective ${\bb Z}[\pi\times G]$-modules. This induces (by Section 3 of \cite{mislin}, for example) an element $[H_{n+1}(f^n;{\bb Z}\pi)]\in \tilde{K}_0({\bb Z}[\pi\times G])$, such that the element vanishes if and only if all the projective modules in the resolution can be chosen to be free. Therefore the element is the obstruction for completing our construction.

We identify this obstruction with the $K$-theory element represented by the ${\bb Z}\pi$-chain complex $C_*(\tilde{f})$, regarded as a ${\bb Z}[\pi\times G]$-chain complex with a trivial $G$-action. This is crucial for detailed calculations.

There is a ${\bb Z}G$-free resolution $P$ of ${\bb Z}$ that is finitely ${\bb Z}G$-generated in each dimension. In addition, there is a ${\bb Z}[\frac{1}{|G|}][G]$-projective resolution $P'$ of ${\bb Z}[\frac{1}{|G|}]$ that is nonzero only in dimensions $0$ and $1$. Then we have
\[
C_*(\tilde{f})=C_*(\tilde{f})\otimes {\bb Z}
\simeq C_*(\tilde{f})\otimes P
\simeq C_*(\tilde{f})\otimes P[\tfrac{1}{|G|}]
\simeq C_*(\tilde{f})\otimes P'.
\]
The first chain homotopy equivalence is due to the resolution. The second chain homotopy equivalence is due to the Smith condition, $H_*(f;{\bb F}_p\pi)=0$, for all prime factors $p$ of $|G|$. The third chain homotopy equivalence is due to the fact that $P[\tfrac{1}{|G|}]$ is also a ${\bb Z}G$-projective resolution of ${\bb Z}[\frac{1}{|G|}]$. We note that each term in $C_*(\tilde{f})\otimes P$ is finitely generated and ${\bb Z}[\pi\times G]$-free. We also note that, since $C_*(\tilde{f})$ vanishes above dimension $n$, and $P'$ vanishes away from dimensions $0$ and $1$, the chain complex $C_*(\tilde{f})\otimes P'$ has cohomological dimension $\le n+1$. This implies (by Theorem 3.5 of \cite{mislin} or Lemma 1.7 of \cite{av}, for example) that $C_*(\tilde{f})$ is chain homotopy equivalent to a finite chain complex of finitely generated, projective ${\bb Z}[\pi\times G]$-modules. This gives a well defined element $[C_*(\tilde{f})]\in \tilde{K}_0({\bb Z}[\pi\times G])$.

The chain complex $C_*(\tilde{f})$ fits into an exact sequence of ${\bb Z}[\pi\times G]$-chain complexes
\[
0\to C_*(\tilde{f})
\to C_*(\tilde{f}^n)
\to C_{*-1}(\tilde{X}^n,\tilde{F})
\to 0.
\]
Since $C_{*-1}(\tilde{X}^n,\tilde{F})$ is a finite chain complex of finitely generated ${\bb Z}[\pi\times G]$-free modules, $C_*(\tilde{f}^n)$ is also chain homotopy equivalent to a finite chain complex of finitely generated, projective ${\bb Z}[\pi\times G]$-modules, and $[C_*(\tilde{f}^n)]=[C_*(\tilde{f})]\in \tilde{K}_0(\pi\times G)$. On the other hand, we know the homology of $C_*(\tilde{f}^n)$ vanishes at all dimensions except for $H_{n+1}(f^n;{\bb Z}\pi)$. Therefore the ${\bb Z}[\pi\times G]$-chain complex $\cdots\to 0\to H_{n+1}(f^n;{\bb Z}\pi)\to 0\to \cdots$ is chain homotopy equivalent to $C_*(\tilde{f}^n)$. This implies that $H_{n+1}(f^n;{\bb Z}\pi)$ has a finite resolution by finitely generated free ${\bb Z}[\pi\times G]$-modules, and
\[
(-1)^{n+1}[H_{n+1}(f^n;{\bb Z}\pi)]
=[C_*(\tilde{f}^n)]
=[C_*(\tilde{f})]
\in \tilde{K}_0({\bb Z}[\pi\times G]).
\]
The equality to $[C_*(\tilde{f})]$ shows that the obstruction $(-1)^{n+1}[H_{n+1}(f^n;{\bb Z}\pi)]$ is independent of our choice of construction.
\end{proof}

\begin{proof}[Proof of Theorem \ref{nontrivialtarget}] 
The proof is similar to Theorem \ref{trivialtarget}. The inductive construction of $f^n\colon X^n\to Y$ is the same, except the new cells can be mapped to $Y$ instead of just the fixed set, and all the homotopy groups and homology groups (at the universal cover level) are ${\bb Z}[\Gamma]$-modules. For sufficiently large $n$, the obstruction for constructing the pseudo-equivalence is $[H_{n+1}(f^n;{\bb Z}\pi)]\in \tilde{K}_0({\bb Z}[\Gamma])$. We need to argue that $C_*(\tilde{f})$ represents an element in $\tilde{K}_0({\bb Z}[\Gamma])$ that is the same as $(-1)^{n+1}[H_{n+1}(f^n;{\bb Z}\pi)]$.

In Section \ref{scondition}, we saw that the global Smith condition \eqref{globalsmith} in Theorem \ref{nontrivialtarget} is equivalent to the local Smith condition \eqref{localsmith} for each component $C$ of $Y^G$. Following the same argument for Theorem \ref{trivialtarget}, we know the ${\bb Z}[\pi_{\hat{C}}\times \Gamma_{\hat{C}}]$-chain complex $C_*(\hat{F}_C\to \hat{C})$ is chain homotopy equivalent to a finite chain complex of finitely generated, projective ${\bb Z}[\pi_{\hat{C}}\times \Gamma_{\hat{C}}]$-modules. This gives a well-defined $K$-theory element
\[
[C_*(\hat{F}_C\to \hat{C})]
\in \tilde{K}_0({\bb Z}[\pi_{\hat{C}}\times \Gamma_{\hat{C}}]).
\]
By $\pi_{\hat{C}}\times \Gamma_{\hat{C}}\sub \pi\rtimes \Gamma_{\hat{C}}=\Gamma$, this inducts to
\begin{align*}
[C_*(\tilde{F}_C\to \tilde{C})]
&=[{\bb Z}[\Gamma]\otimes_{{\bb Z}[\pi_{\hat{C}}\times \Gamma_{\hat{C}}]}C_*(\hat{F}_C\to \hat{C})] \\
&=\text{Ind}_{{\bb Z}[\pi_{\hat{C}}\times \Gamma_{\hat{C}}]}^{{\bb Z}[\Gamma]}[C_*(\hat{F}_C\to \hat{C})]
\in \tilde{K}_0({\bb Z}[\Gamma]),
\end{align*}
which further adds up to 
\[
[C_*(\tilde{f})]
=\sum_{C\in \pi_0Y^G}[C_*(\tilde{F}_C\to \tilde{C})]
\in \tilde{K}_0({\bb Z}[\Gamma]).
\]

Now we know the ${\bb Z}[\Gamma]$--chain complex $C_*(\tilde{f})$ is chain homotopy equivalent to a finite chain complex of finitely generated, projective ${\bb Z}[\Gamma]$-modules and gives a $K$-theory element. Then we have short exact sequences of ${\bb Z}[\Gamma]$-chain complexes
\begin{align*}
0&\to C_*(\tilde{f}\colon \tilde{F}\to \widetilde{Y^G})
\to C_*(\tilde{F}\to \widetilde{Y})
\to C_{*-1}(\widetilde{Y},\widetilde{Y^G})
\to 0, \\
0&\to C_*(\tilde{F}\to \widetilde{Y})
\to C_*(\tilde{f}^n\colon \tilde{X}^n\to \widetilde{Y})
\to C_{*-1}(\tilde{X}^n,\tilde{F})
\to 0.
\end{align*}
Since the $G$-action on $Y$ is semi-free, the ${\bb Z}[\Gamma]$-modules in $C_{*-1}(\widetilde{Y},\widetilde{Y^G})$ are free. Since $X^n$ is obtained by glueing free $G$-cells to $F$, the ${\bb Z}[\Gamma]$-modules in $C_{*-1}(\tilde{X}^n,\tilde{F})$ are also free. Therefore
\begin{align*}
[C_*(\tilde{f})]
&=[C_*(\tilde{F}\to \widetilde{Y})]  \\
&=[C_*(\tilde{f}^n\colon \tilde{X}^n\to \widetilde{Y})]  \\
&=(-1)^{n+1}[H_{n+1}(f^n;{\bb Z}\pi)]
\in \tilde{K}_0({\bb Z}[\Gamma]).
\end{align*}
This completes the identification of the $K$-theory obstruction. 
\end{proof}

The main theorems can be modified to get the pseudo-equivalence extension to be a {\em simple} homotopy equivalence. The only change is the $K$-theory in which the obstruction lives. The following is the analogue of Theorem \ref{trivialtarget} making use of an algebraic $K$-theory introduced in \cite{av}.

\begin{theorem}\label{jones_simple}
Suppose $f\colon F\to Y$ is a map of finite complexes, with $Y$ connected and $\pi=\pi_1Y$. Then $F$ can be the fixed set of a finite semi-free $G$-complex $X$, and $f$ has simple pseudo-equivalence extension $g\colon X\to Y$, if and only if the following are satisfied.
\begin{enumerate}
\item The map $f$ induces isomorphisms $\tilde{H}_*(F;{\bb F}_p\pi)\cong \tilde{H}_*(Y;{\bb F}_p\pi)$ for all prime factors $p$ of $|G|$. 
\item An obstruction $[C_*(\tilde{f})]\in Wh_1^T(\pi\subset\pi\times G)$ vanishes.
\end{enumerate}
\end{theorem}

Assadi and Vogel \cite{av} introduced the Grothendick group $Wh_1^T(\pi\subset\pi\times G)$ of the additive category of finitely generated ${\bb Z}\pi$-based ${\bb Z}[\pi\times G]$-projective modules, such that the ${\bb Z}[\pi\times G]$-projective (actually free) module ${\bb Z}[\pi\times G]$ with the choice of $G$ as ${\bb Z}\pi$-basis is trivial in $Wh_1^T$. They showed that there is an exact sequence ($\tilde{K}_0(\pi)$ is denoted $Wh_0(\pi)$ in \cite{av})
\[
Wh_1(\pi\times G)\overset{T}{\to}
Wh_1(\pi)\overset{\beta}{\to}
Wh_1^T(\pi\subset\pi\times G)\overset{\alpha}{\to} 
\tilde{K}_0({\bb Z}[\pi\times G])\overset{T}{\to} 
\tilde{K}_0({\bb Z}\pi).
\]
The cells of $F$ and $Y$ give natural ${\bb Z}\pi$-bases for the modules in $C_*(\tilde{f})$. By Lemmas 1.6 and 1.7 of \cite{av}, under the Smith condition, the chain complex $C_*(\tilde{f})$ with the natural ${\bb Z}\pi$-bases gives a well defined element $[C_*(\tilde{f})]\in Wh_1^T(\pi\subset\pi\times G)$. The image of this element in $\tilde{K}_0({\bb Z}[\pi\times G])$ is the obstruction in the main theorem.

The proof of Theorem \ref{jones_simple} is the same as the proof of Theorem \ref{trivialtarget}, with additional tracking of the basis in the construction. The key point is that the free $G$-cells used in the construction give the chain complex $C_{*-1}(\tilde{X}^n,\tilde{F})$, where each term is the direct sum of finitely many copies of ${\bb Z}[\pi\times G]$ with the ${\bb Z}\pi$-basis $G$. Therefore $[C_{*-1}(\tilde{X}^n,\tilde{F})]=0\in Wh_1^T(\pi\subset\pi\times G)$.

The next result is the modification needed to treat compact ANR-spaces in place of finite complexes.

\begin{theorem}\label{jones_anr}
Suppose $f\colon F\to Y$ is a map of compact ANR-spaces, with $Y$ connected and $\pi=\pi_1Y$. Then $F$ can be the fixed set of a semi-free compact $G$-ANR-space $X$, and $f$ has pseudo-equivalence extension $g\colon X\to Y$, if and only if the following are satisfied.
\begin{enumerate}
\item The map $f$ induces isomorphisms $\tilde{H}_*(F;{\bb F}_p\pi)\cong \tilde{H}_*(Y;{\bb F}_p\pi)$ for all prime factors $p$ of $|G|$. 
\item An obstruction $[C_*(\tilde{f})]\in \tilde{K}_0^{top}({\bb Z}\pi\subset {\bb Z}[\pi\times G])$ vanishes.
\end{enumerate}
\end{theorem}

The topological $K$-theory $\tilde{K}_0^{top}$ was introduced by M. Steinberger and J. West \cite{steinberger,sw}, and by F. Quinn \cite{quinn1,quinn2} as the $K$-theoretical obstruction for the topological version of the finiteness theorems for $G$-ANR-spaces. It fits into an exact sequence
\begin{align*}
& H_1(\pi;\tilde{\bf K}({\bb Z}[G]))\to
Wh(\pi\times G)\to 
Wh^{top}(\pi\subset\pi\times G) \\
\to &
H_0(\pi;\tilde{\bf K}({\bb Z}[G]))\to
\tilde{K}_0({\bb Z}[\pi\times G])\to 
\tilde{K}_0^{top}({\bb Z}\pi\subset {\bb Z}[\pi\times G])
\to\cdots
\end{align*}

By a theorem of West \cite{west}, compact ANRs are homotopy equivalent to finite complexes. Therefore we have an obstruction $[C_*(\tilde{f})]\in \tilde{K}_0({\bb Z}[\pi\times G])$ in our main pseudo-equivalence extension theorem. Quinn \cite{quinn1} showed that controlled finitely dominated complexes over $F$ with free $G$-actions have controlled Wall finiteness obstructions in $H_0(F;\tilde{\bf K}({\bb Z}[G]))$, and all elements of this group can be realized. By glueing on such an element, we can change the obstruction $[C_*(\tilde{f})]\in \tilde{K}_0({\bb Z}[\pi\times G])$ by any element in the image of $H_0(F;\tilde{\bf K}({\bb Z}[G]))$. We will explain in the next paragraph that the natural map $H_0(F;\tilde{\bf K}({\bb Z}[G]))\to H_0(Y;\tilde{\bf K}({\bb Z}[G]))\to H_0(\pi;\tilde{\bf K}({\bb Z}[G]))$ is surjective. Therefore $H_0(F;\tilde{\bf K}({\bb Z}[G]))$ and $H_0(\pi;\tilde{\bf K}({\bb Z}[G]))$ have the same image in $\tilde{K}_0({\bb Z}[\pi\times G])$. Thus we conclude that the obstruction for $G$-ANR pseudo-equivalence extension problem actually lies in the image of $\tilde{K}_0({\bb Z}[\pi\times G])$ in $
\tilde{K}_0^{top}({\bb Z}\pi\subset {\bb Z}[\pi\times G])$.

Carter's vanishing theorem \cite{carter2} says that $K_{-i}({\bb Z}[G])=0$ for finite group $G$ and $i>1$. Therefore the spectral sequence that computes $H_0(F;\tilde{\bf K}({\bb Z}[G]))$ consists of only $H_0(F;\tilde{K}_0({\bb Z}[G]))$ and $H_1(F;\tilde{K}_{-1}({\bb Z}[G]))$. The same is true for $H_0(Y;\tilde{\bf K}({\bb Z}[G]))$ and $H_0(\pi;\tilde{\bf K}({\bb Z}[G]))$. To show the surjection, therefore, we only need to show that $H_i(F;\tilde{K}_{-i}({\bb Z}[G]))\to H_i(Y;\tilde{K}_{-i}({\bb Z}[G]))\to H_i(\pi;\tilde{K}_{-i}({\bb Z}[G]))$ is surjective for $i=0,1$. The Smith condition implies that $\pi_iF\to \pi_iY$ is surjective for $i=0,1$. This implies the surjections on $H_i(?;\tilde{K}_{-i}({\bb Z}[G]))$ for $i=0,1$.

\section{Calculations and Examples}

Let $T(r)$ be the mapping torus of a map $S^d\to S^d$ of degree  $r$. Let
\[
f\colon F=T(r)\to Y=S^1.
\]
be the projection map. For a finite group $G$ of order $n$, we try to extend $F$ to be the fixed set of a finite semi-free $G$-complex $X$, and extend $f$ to a pseudo-equivalence $g\colon X\to S^1$. 

We have $\pi_1Y=\langle t\rangle=\{t^i\colon i\in {\bb Z}\}\cong {\bb Z}$, and the only non-trivial ${\bb Z}\langle t\rangle$-homology\footnote{In the literature, ${\bb Z}\langle t\rangle$ is usually denoted ${\bb Z}[t,t^{-1}]$. We use ${\bb Z}\langle t\rangle$ to simplify notation.} of $f$ is
\[
H_d(f;{\bb Z}\langle t\rangle)
={\bb Z}\langle t\rangle/(rt-1).
\]
For a prime $p$, we have $H_d(f;{\bb Z}_p\langle t\rangle)=0$ if and only if $p|r$. Therefore the Smith condition is satisfied for $G$ if and only if
\[
p|n\implies p|r.
\]
This is equivalent to $n$ dividing some power of $r$. Under this assumption, the condition for the semi-free pseudo-equivalence extension is the vanishing of
\[
[{\bb Z}\langle t\rangle/(rt-1)]
\in \tilde{K}_0({\bb Z}[G]\langle t\rangle).
\]

\begin{proposition}\label{eg1}
If $G$ be a finite group of order $n$, and $r$ is a multiple of $n$, then $T(r)\to S^1$ has semi-free pseudo-equivalence extension. If $G$ is also abelian, then the converse is also true.
\end{proposition}

We note that, when $G$ is abelian and $r$ is not a multiple of $n$, the counterexample constructed in the proof below actually has non-vanishing obstruction in $NK_0({\bb Z}[G])$. In fact, by Theorem \ref{jones_anr}, the map does not even have semi-free pseudo-equivalence extension in the $G$-ANR category.

\begin{proof}
We have the Bass-Heller-Swan decompositions
\begin{align}
K_1({\bb Z}_n\langle t\rangle)
&= K_1({\bb Z}_n)\oplus 
K_0({\bb Z}_n)\oplus 
NK_1({\bb Z}_n)\oplus 
NK_1({\bb Z}_n), \label{k1} \\
\tilde{K}_0({\bb Z}[G]\langle t\rangle)
&= \tilde{K}_0({\bb Z}[G])\oplus 
K_{-1}({\bb Z}[G])\oplus 
NK_0({\bb Z}[G])\oplus 
NK_0({\bb Z}[G]). \label{k0} 
\end{align}
We also have the pullbacks of rings \cite{milnor} ($\Sigma_G=\sum_{g\in G}g$)
\begin{equation}\label{square}
\begin{CD}
{\bb Z}[G] @>>> 
{\bb Z}[G]/\Sigma_G \\
@VVV @VVV \\
{\bb Z}  @>>> {\bb Z}_n
\end{CD}\quad\quad\quad
\begin{CD}
{\bb Z}[G]\langle t\rangle @>>> 
({\bb Z}[G]/\Sigma_G)\langle t\rangle \\
@VVV @VVV \\
{\bb Z}\langle t\rangle  @>>> {\bb Z}_n\langle t\rangle
\end{CD}
\end{equation}
that induce the Swan homomorphisms\footnote{According to the seminal paper \cite{swan} of Swan where this construction first arose.} $\partial\colon K_1({\bb Z}_n\langle t\rangle)\to \tilde{K}_0({\bb Z}[G]\langle t\rangle)$ that are compatible with the Bass-Heller-Swan decompositions. We note that our obstruction is an image of the Swan homomorphism
\[
[{\bb Z}\langle t\rangle/(rt-1)]=\partial[rt-1],\quad
[rt-1]\in K_1({\bb Z}_n\langle t\rangle).
\]
Since $n$ dividing $r$ implies $[rt-1]=0$, we get the sufficient part of the proposition. 

For the necessary part, we note that the Smith condition requires $n$ dividing some power of $r$. Now we identify the obstruction in the Bass-Heller-Swan decomposition. By the proof of the decomposition in \cite[Theorem 3.2.22]{ro}, we write the automorphism of $R\langle t\rangle$ as an automorphism of $R$ with a nilpotent correction
\[
rt-1
=r(t-1)+(r-1)
=(r-1)[1+(r-1)^{-1}r(t-1)]
\in {\bb Z}_n\langle t\rangle,
\]
where $r-1$ is invertible and $(r-1)^{-1}r$ is nilpotent by the Smith condition. This shows that the element $[rt-1]$ becomes $([r-1],0,0,[(r-1)^{-1}r])$ in the decomposition, and our obstruction is $\partial[r-1]\in \tilde{K}_0({\bb Z}[G])$ and $\partial [(r-1)^{-1}r]\in NK_0({\bb Z}[G])$. We will concentrate on the vanishing of $\partial [(r-1)^{-1}r]$, which is the image of $(r-1)^{-1}(rt-1)\in K_1({\bb Z}_n\langle t\rangle)$ under the Swan homomorphism. 

Now we assume $G$ is abelian. Then we have the determinant maps from $K_1$ to the groups of invertible elements. The Swan homomorphism is part of an exact sequence compatible with the determinants
\[
\begin{CD}
K_1({\bb Z}\langle t\rangle)\oplus 
K_1(({\bb Z}[G]/\Sigma_G)\langle t\rangle)
@>{\alpha}>>  
K_1({\bb Z}_n\langle t\rangle)
@>{\partial}>>  
\tilde{K}_0({\bb Z}[G]\langle t\rangle) \\
@VV{\det}V @VV{\det}V \\
{\bb Z}\langle t\rangle^*\oplus 
({\bb Z}[G]/\Sigma_G)\langle t\rangle^*  
@>{\beta}>> 
{\bb Z}_n\langle t\rangle^*
\end{CD}
\]
The vanishing of $\partial [(r-1)^{-1}r]$ implies that $[(r-1)^{-1}(rt-1)]$ is in the image of $\alpha$. This further implies that $\det[(r-1)^{-1}(rt-1)]=(r-1)^{-1}(rt-1)$ is in the image of $\beta$. In the appendix of this paper, we prove that $({\bb Z}[G]/\Sigma_G)\langle t\rangle^*=({\bb Z}[G]/\Sigma_G)^*\langle t\rangle$. In other words, the invertibles in $({\bb Z}[G]/\Sigma_G)\langle t\rangle$ are monomials. If $n$ does not divide $r$, then $(r-1)^{-1}(rt-1)$ is not a monomial, and we get a contradiction. This proves the necessary part. 
\end{proof}

Example \ref{example1} in the introduction is a direct consequence of Proposition \ref{eg1}. In fact, $T(p)$ is not only not the fixed set of a semi-free ${\bb Z}_{p^2}$-action on homotopy circle, it is also not the fixed set of a semi-free ${\bb Z}_p\times{\bb Z}_p$-action. The same argument given above shows that $T(p^2)$ is not fixed under a semi-free ${\bb Z}_p\times{\bb Z}_{p^2}$-action or ${\bb Z}_p\times{\bb Z}_p\times{\bb Z}_p$-action, etc.

For the sufficiency part of Proposition \ref{eg1}, we give an explicit construction for the special case that $G$ acts freely on a sphere. For example, $G$ is a cyclic group acting on the circle $S^1$ by the standard rotations.

\begin{example}
Suppose $r$ is a multiple of $n=|G|$, and $G$ acts freely on a sphere $S^e$. By replacing $S^e$ with $S^e*S^e=S^{2e+1}$ and taking the join of $G$-actions, we may further assume that the action preserves the orientation of $S^e$. Consider the join $S^{d+e+1}=S^d*S^e$, with the trivial $G$-action on $S^d$ and the given $G$-action on $S^e$. The action is semi-free with fixed set $(S^{d+e+1})^G=S^d$. Let $h$ be a self map of $S^{d+e+1}$ that is the join of the degree $r$ map on $S^d$ and the identity map on $S^e$. Then $h$ is a $G$-map of degree $r$. For any free point $x\in S^{d+e+1}-S^d$, let $D$ be a small disk around $x$, such that the action of $G$ on $D$ gives disjoint copies. By shrinking the boundary $\partial D$ of $D$ to the point $x$, we get a map $S^{d+e+1}\to S^{d+e+1}\vee_x(D/\partial D)$. Combining the identity on $S^{d+e+1}$ and a homeomorphism $D/\partial D\to S^{d+e+1}$, we get a map $S^{d+e+1}\to S^{d+e+1}\vee_xS^{d+e+1}\to S^{d+e+1}$. If we do this for all $G$ copies of $D$, then we get a $G$-map $h'\colon S^{d+e+1}\to S^{d+e+1}\vee_{Gx}G(D/\partial D)\to S^{d+e+1}$. By choosing suitable homeomorphism $D/\partial D\to S^{d+e+1}$, the degree of $h'$ is $r+n$ or $r-n$. By repeating the construction for several points in $S^{d+e+1}-S^d$, we get a $G$-map $h''\colon S^{d+e+1}\to S^{d+e+1}$ of degree $r+an$ for any integer $a$. Since $r$ is a multiple of $n$, we take $a=-\frac{r}{n}$ and get a $G$-map $h''\colon S^{d+e+1}\to S^{d+e+1}$ of degree $0$. On the other hand, the modification happens only on the free part of $S^{d+e+1}$. Therefore the restriction of $h''$ on the fixed part is still the original degree $r$ map $S^d\to S^d$. The mapping torus $T(h'')$ has semi-free $G$-action with fixed set $T(r)$, and $T(h'')\to S^1$ extends $T(r)\to S^1$. Moreover, since the degree of $h''$ is $0$, the map $T(h'')\to S^1$ is a homotopy equivalence.
\end{example}

Now we turn to another application showing other phenomena. Suppose $n$ is not a prime power. Then $n=n_1n_2$, with $n_1,n_2>1$ and coprime. Let $a$ satisfy $a=1$ mod $n_1$ and $a=0$ mod $n_2$. Then $b=1-a$ satisfies $b=0$ mod $n_1$ and $b=1$ mod $n_2$. Let $T(a,b)$ be the  double torus of two maps of $S^d$ to itself of respective degrees $a,b$. Let $f\colon T(a,b)\to S^1$ be the natural map. We consider the pseudo-equivalence extension of $f$ for the action by the cyclic group $G={\bb Z}_n$.

Similar to the mapping torus in the earlier example, the only non-trivial ${\bb Z}\langle t\rangle$-homology of $f$ is 
\[
H_d(f;{\bb Z}\langle t\rangle)
={\bb Z}\langle t\rangle/(at-b).
\]
Since $(at-b)(at^{-1}-b)=a^2=1$ mod $n_1$ and $(at-b)(at^{-1}-b)=b^2=1$ mod $n_2$, we have $at-b$ invertible in ${\bb Z}_n\langle t\rangle$. This verifies the Smith condition for $f$. 

\begin{proposition}\label{eg2}
If $n$ is not prime power order, then for suitable $a,b$, the map $T(a,b)\to S^1$ satisfies the Smith condition for the cyclic group $G={\bb Z}_n$, but has no semi-free pseudo-equivalence extension. 
\end{proposition}

For the counterexample constructed in the proof below, the obstruction effectively lies in the direct summand $K_{-1}({\bb Z}[{\bb Z}_n])\sub \tilde{K}_0({\bb Z}[{\bb Z}_n]\langle t\rangle)$ according to \eqref{k0}. By Theorem \ref{jones_anr}, although the map has no semi-free pseudo-equivalence extension in the $G$-complex category, it does have semi-free pseudo-equivalence extension in the $G$-ANR category.

\begin{proof}
The obstruction for pseudo-equivalence extension is 
\[
[{\bb Z}\langle t\rangle/(at-b)]\in \tilde{K}_0({\bb Z}[{\bb Z}_n]\langle t\rangle).
\]
This is the image of 
\[
[at-b]\in {\bb Z}_n\langle t\rangle^* 
\subset K_1({\bb Z}_n\langle t\rangle)
\]
under the Swan homomorphism. We carry out the argument similar to Proposition \ref{eg1}. 

We have $a(1-a)=ab=0$ mod $n$. This means $a^2=a$ mod $n$, or $a$ is an idempotent mod $n$. In particular, $a{\bb Z}_n$ is a projective ${\bb Z}_n$-module. By the proof of \cite[Theorem 3.2.22]{ro}, the obstruction $[at-b]$ on the left of \eqref{k1} corresponds to $(0,[a{\bb Z}_n],0,0)$ on the right. Therefore our obstruction is the image of $[a{\bb Z}_n]\in K_0({\bb Z}_n)$ (as in \eqref{k0}) under the Swan homomorphism $K_0({\bb Z}_n)\to K_{-1}({\bb Z}[{\bb Z}_n])$ induced by the pullback \eqref{square}. By the calculation of \cite{bm}, this element is a non-divisible element of $K_{-1}({\bb Z}[{\bb Z}_n])$. 

Then the Swan homomorphism fits into an exact sequence 
\begin{align*}
&\tilde{K}_0({\bb Z})\oplus \tilde{K}_0({\bb Z}[{\bb Z}_n]/\Sigma_{{\bb Z}_n})
\to \tilde{K}_0({\bb Z}_n) \to K_{-1}({\bb Z}[{\bb Z}_n]) \\
\to &
K_{-1}({\bb Z})\oplus K_{-1}({\bb Z}[{\bb Z}_n]/\Sigma_{{\bb Z}_n})
\to K_{-1}({\bb Z}_n).
\end{align*}
By \cite{bm,se}, we have $\tilde{K}_0({\bb Z})=K_{-1}({\bb Z})=0$, $\tilde{K}_0({\bb Z}[{\bb Z}_n]/\Sigma_{{\bb Z}_n})$ is finite, and $\tilde{K}_0({\bb Z}_n)$ and $K_{-1}({\bb Z}[{\bb Z}_n]/\Sigma_{{\bb Z}_n})$ are free abelian. Therefore $\tilde{K}_0({\bb Z}_n)$ embeds as a direct summand of $K_{-1}({\bb Z}[{\bb Z}_n])$.

If $n=p_1^{m_1}\dots p_l^{m_l}$ is the decomposition into distinct primes, then
\[
K_0({\bb Z}_n)
=\oplus_{i=1}^l K_0({\bb Z}_{p_i^{m_i}})
=\oplus_{i=1}^l {\bb Z}.
\] 
We note that the projective ${\bb Z}_n$-module $a{\bb Z}_n$ is isomorphic to ${\bb Z}_{n_1}$. Under the isomorphism $K_0({\bb Z}_n)=K_0({\bb Z}_{n_1})\oplus K_0({\bb Z}_{n_2})$, $[a{\bb Z}_n]\in K_0({\bb Z}_n)$ corresponds to $([{\bb Z}_{n_1}],0)\in K_0({\bb Z}_{n_1})\oplus K_0({\bb Z}_{n_2})$. If we start by choosing $n_1=p_1^{m_1}$ and $n_2=\frac{n}{n_1}$, then the obstruction for $f\colon T(a,b)\to S^1$ is the image of $(1,0,\dots,0)\in K_0({\bb Z}_n)$ under the injective Swan homomorphism. Similarly, we can make other choices of $n_1,n_2$, such that the obstructions for the corresponding $f\colon T(a,b)\to S^1$ are the images of other ``unit vectors'' in $K_0({\bb Z}_n)$. The upshot is that, if $n$ is not a prime power, then we can construct $f\colon F\to S^1$ satisfying the Smith condition, and the obstruction is a nonzero element in $K_{-1}({\bb Z}[{\bb Z}_n])$.
\end{proof}

The counterexamples for Propositions \ref{eg1} and \ref{eg2} can fit into other spaces.

\begin{theorem}\label{fj1}
Suppose $Y$ is a finite complex with torsionless $\pi=\pi_1Y$. Suppose the Farrell-Jones conjecture holds for $\pi$. Suppose $G$ is finite cyclic, and $|G|$ is not a prime and has no square factor. Then there is a map $F\to Y$ satisfying the Smith condition, but has no semi-free pseudo-equivalent extension if and only if $H_1\pi\ne 0$.
\end{theorem}

The condition on the order of $G$ is that $|G|$ is a product of more than one distinct primes.

\begin{proof}
The hypothesis that $|G|$ has no square factor implies that $NK_0({\bb Z}[G])=0$. Then for torsionless $\pi$, the Farrell-Jones conjecture asserts that
\[
K_0({\bb Z}[\pi\times G])
=K_0({\bb Z}[G])\oplus H_1(\pi;K_{-1}({\bb Z}[G])).
\]
If $H_1\pi=0$, then $H_1(\pi;K_{-1}({\bb Z}[G]))=0$, and the Swan homomorphism relevant to constructing the $G$-action lies in the $K_0({\bb Z}[G])$ part. Thus we are reduced to the classical Swan homomorphism. By \cite[Corollary 6.1]{swan}, the Swan homomorphism vanishes for cyclic $G$.

If $H_1\pi\ne 0$, then a generator of $H_1\pi=H_1Y$ can be represented by a loop $S^1\to Y$. By Proposition \ref{eg2}, there is a map $f\colon F\to S^1$ satisfying the Smith condition but has non-vanishing pseudo-equivalent extension obstruction $[C(\tilde{f})]\in \tilde{K}_{-1}({\bb Z}[G])=H_1(S^1,\tilde{K}_{-1}({\bb Z}[G]))\sub \tilde{K}_0({\bb Z}[G]\langle t\rangle)$, where $t$ is the generator of $\pi_1S^1$. We use the loop $S^1\to Y$ to extend $f$ to $f'\colon F\cup_{S_1}Y\to Y$. Then $f'$ also satisfies the Smith condition, and the pseudo-equivalent extension obstruction $[C(\tilde{f'})]$ for $f'$ is the image of $[C(\tilde{f})]$ under the homomorphism
\[
\tilde{K}_{-1}({\bb Z}[G])=H_1(S^1,\tilde{K}_{-1}({\bb Z}[G]))
\to H_1(\pi,\tilde{K}_{-1}({\bb Z}[G])).
\]
Since the circle represents to a generator of $H_1\pi$, the image obstruction is still nonzero. 
\end{proof}

\begin{theorem}\label{fj2}
Suppose $Y$ is a finite complex with torsionless $\pi=\pi_1Y$, the Farrell-Jones conjecture holds for $\pi$, and $\pi$ has maximal infinite cyclic subgroup $C$, such that the normaliser of $C$ is $C$ itself. Suppose $G$ is a finite abelian group, such that $|G|$ has square factor. Then there is a map $F\to Y$ satisfying the Smith condition, but has no semi-free pseudo-equivalent extension.
\end{theorem}

The theorem implies that, if every a map $F\to Y$ satisfying the Smith condition has semi-free pseudo-equivalent extension for ${\bb Z}_n$-action, then $n$ is a product of at least two distinct primes.

\begin{proof}
The proof is similar to Theorem \ref{fj1}, except that the Farrell-Jones conjecture is more complicated because $NK_0({\bb Z}[G])\ne 0$. In this case, by \cite{bartels}, the formula for $K_0({\bb Z}[\pi\times G])$ has another factor (i.e., a direct summand), namely $H_*^{\pi\times G}(E_{{\mc V}{\mc C}}(\pi\times G), E_{{\mc F}{\mc I}{\mc N}}(\pi\times G); {\bf K})$. Here $H_*^{\pi\times G}$ is the homology over the category of $\pi\times G$-orbits by Davis and L\"uck \cite{davis-luck}, $E_{{\mc V}{\mc C}}$ is the classifying space for the family of virtually cyclic subgroups, $E_{{\mc F}{\mc I}{\mc N}}$ is the classifying space for finite subgroups, and ${\bf K}$ is the non-connective $K$-theory spectrum of the isotropy groups of points. Note that this relative homology is concentrated on points with infinite isotropy. Under the condition of $C$ normalised only by itself, by \cite{fj}, the set of points with isotropy $C$ contributes two copies of $NK_0({\bb Z}[G])$, and the glueing trick using our example from Proposition \ref{eg1} constructs an obstructed example. 

Since this element is nonzero in $\tilde{K}_0^{G,top}$ in the sense of \cite{quinn2,steinberger}, it even obstructs the existence of an ANR-action.
\end{proof}

Finally, we study the pseudo-equivalence extension of a map between $3$-dimensional lens spaces. Denote $L_3(q)=L_3(q;1,1)$.

\begin{proposition}\label{lens}
Suppose $f\colon L_3(kp) \to L_3(p)$ is a degree $d$ map, where $p$ is a prime not dividing $d$ and $k$. Then $f$ has a pseudo-equivalent for ${\bb Z}_p$-action if and only if $d^{p-1}=k^{p-1}$ mod $p^2$.
\end{proposition}

\begin{proof}
The obstruction lies in $\tilde{K}_0({\bb Z}[{\bb Z}_p\times {\bb Z}_p])$, where the first ${\bb Z}_p=\pi_1L_3(p)$, and the second ${\bb Z}_p$ is the action group. The obstruction is given by the chain complex of the map $\tilde{f}\colon \tilde{L}_3(kp)\to \tilde{L}_3(p)$ obtained by pulling back along the universal cover $\tilde{L}_3(p)=S^3\to \tilde{L}_3(p)$. We have the long exact sequence 
\begin{align*}
& H_3(L_3(kp);{\bb Z}[{\bb Z}_p])={\bb Z}\to 
H_3(L_3(p);{\bb Z}[{\bb Z}_p])={\bb Z}\to 
H_3(f;{\bb Z}[{\bb Z}_p])={\bb Z}_d \\
\to &
H_2(L_3(kp);{\bb Z}[{\bb Z}_p])=0\to 
H_2(L_3(p);{\bb Z}[{\bb Z}_p])=0\to 
H_2(f;{\bb Z}[{\bb Z}_p])={\bb Z}_k \\
\to &
H_1(L_3(kp);{\bb Z}[{\bb Z}_p])={\bb Z}_{kp}\to 
H_1(L_3(p);{\bb Z}[{\bb Z}_p])={\bb Z}_p\to 
H_1(f;{\bb Z}[{\bb Z}_p])=0.
\end{align*}
We note that $H_*(f;{\bb Z}[{\bb Z}_p])=0,0,{\bb Z}_k,{\bb Z}_d,0,\dots$ are trivial ${\bb Z}[{\bb Z}_p\times {\bb Z}_p]$-modules, and have ${\bb Z}[{\bb Z}_p\times {\bb Z}_p]$-projective resolutions. Therefore the Euler characteristic of $H_*(f;{\bb Z}[{\bb Z}_p])$ gives an element of $K_1({\bb Z}_{p^2})$ ($p^2$ is the order of the group ${\bb Z}_p\times {\bb Z}_p$), and the obstruction is the image of this element under the Swan homomorphism for the group ${\bb Z}_p\times {\bb Z}_p$
\[
K_1({\bb Z}_{p^2})=({\bb Z}_{p^2})^*\to 
\tilde{K}_0({\bb Z}[{\bb Z}_p\times {\bb Z}_p]).
\]
In the multiplicative group $({\bb Z}_{p^2})^*$, the Euler characteristic of $H_*(f;{\bb Z}[{\bb Z}_p])$ is $\frac{k}{d}$. The group $({\bb Z}_{p^2})^*$ is additively isomorphic to ${\bb Z}_p\oplus {\bb Z}_{p-1}$. By \cite[Proposition 3]{ullom}, the image of the Swan homomorphism is an additive group ${\bb Z}_p\sub \tilde{K}_0({\bb Z}[{\bb Z}_p\times {\bb Z}_p])$. Therefore an element $r$ is in the kernel of the Swan homomorphism $({\bb Z}_{p^2})^*\cong {\bb Z}_p\oplus {\bb Z}_{p-1}\to{\bb Z}_p$ if and only if $r^{p-1}=1$. In particular, the pseudo-equivalence extension obstruction $\frac{k}{d}\in ({\bb Z}_{p^2})^*$ for $f$ vanishes if and only if $(\frac{k}{d})^{p-1}=1$ in $({\bb Z}_{p^2})^*$. This gives the condition $d^{p-1}=k^{p-1}$ mod $p^2$ in the proposition. 
\end{proof}

\section{Appendix: Invertibles}

\begin{lemma}\label{app2}
If $G$ is a finite abelian group, then the invertibles of $({\bb Z}[G]/\Sigma_G)\langle t\rangle$ are monomials.
\end{lemma}

\begin{proof}
The proof is based on the fact that $R\langle t\rangle^*=R^*\langle t\rangle$ for an integral domain $R$. In other words, the invertibles in the polynomial ring $R\langle t\rangle$ are monomials. In general, ${\bb Z}[G]/\Sigma_G$ is not an integral domain, but can be detected by sufficiently many homomorphisms to integral domains. Specifically, a character $\lambda\colon G\to \langle \xi_n\rangle\sub {\bb C}$ ($\xi_n=e^{\frac{2\pi i}{n}}$ is the $n$-th root of unity and $\langle \xi\rangle$ is all powers of $\xi$) induces a ring homomorphism $\lambda\colon {\bb Z}[G]\to {\bb Z}[\xi_n]={\bb Z}[s]/\varphi_n(s)$, where $\varphi_n$ is the minimal polynomial of $\xi_n$. This further induces a homomorphism $\lambda\colon {\bb Z}[G]/\Sigma_G\to {\bb Z}[\xi_n]$ unless $\lambda$ is trivial. Then we get a homomorphism of invertibles 
\[
\lambda\colon ({\bb Z}[G]/\Sigma_G)\langle t\rangle^* 
\to {\bb Z}[\xi_n]\langle t\rangle^*
={\bb Z}[\xi_n]^*\langle t\rangle.
\]
The equality is due to the fact that ${\bb Z}[\xi_n]$ is an integral domain. Then for an invertible $x=\sum x_it^i$, $x_i\in {\bb Z}[G]$, on the left, we know $\lambda(x)=\sum \lambda(x_i)t^i\in {\bb Z}[\xi_n]^*\langle t\rangle$ means that there is $i(\lambda)$, such that $\lambda(x)=\lambda(x_{i(\lambda)})t^{i(\lambda)}$ (i.e., $\lambda(x_i)=0$ for all $i\ne i(\lambda)$). 

If we take $\lambda$ to be all non-trivial characters, with the corresponding $n=p^k$ being prime powers, then we get an embedding 
\[
{\bb Z}[G]/\Sigma_G
\sub 
\times_{\lambda}{\bb Z}[\xi_{p^k}].
\]
This induces an embedding of invertibles (for the specific selections of $\lambda$)
\[
({\bb Z}[G]/\Sigma_G)\langle t\rangle^*
\sub 
\times_{\lambda}{\bb Z}[\xi_{p^k}]\langle t\rangle^*
=\times_{\lambda}{\bb Z}[\xi_{p^k}]^*\langle t\rangle.
\]
The embedding of invertibles shows that, if we prove that $i(\lambda)$ is independent of the choice of $\lambda$, then the original invertible $x=\sum x_it^i$ on the left is also a monomial. This proves the lemma.

First, by $\varphi_{p^k}(s)=(1-s^{p^k})/(1-s^{p^{k-1}})=1+s^{p^{k-1}}+s^{2p^{k-1}}+\dots+s^{(p-1)p^{k-1}}$, we have a homomorphism
\[
\mu(s)=1\colon
{\bb Z}[\xi_{p^k}]
={\bb Z}[s]/\varphi_{p^k}(s) 
\to 
{\bb F}_p.
\]
Then we have the composition $\mu\circ\lambda\colon {\bb Z}[G]/\Sigma_G\to {\bb F}_p$ that sends every group element to $1$. In particular, the induced map $\mu\circ\lambda\colon ({\bb Z}[G]/\Sigma_G)\langle t\rangle\to {\bb F}_p\langle t\rangle$ depends only on $p$. Then $\mu\circ\lambda(x)=\mu(x_{i(\lambda)})t^{i(\lambda)}$ depends only on $p$. This implies that, if two characters $\lambda$ and $\lambda'$ correspond to the same prime, then $i(\lambda)=i(\lambda')$.

It remains to show that, if $\lambda$ and $\lambda'$ correspond to $\xi_p$ and $\xi_q$, where $p,q$ are distinct primes, then $i(\lambda)=i(\lambda')$. By what we proved above, we only need to verify for any one pair $\lambda$ and $\lambda'$ corresponding to $\xi_p$ and $\xi_q$. Consider a character $\Lambda\colon G\to C=\langle \xi_{pq}\rangle\sub {\bb C}$. Then we have
\[
\begin{CD}
{\bb Z}[G]/\Sigma_G  @. \\
@V{\Lambda}VV \\
{\bb Z}[C]/\Sigma_C={\bb Z}[s]/(\frac{1-s^{pq}}{1-s}) 
@>>> 
{\bb Z}[\xi_{pq}]={\bb Z}[s]/\varphi_{pq}(s) \\
@VVV @VVV \\
{\bb Z}[\xi_p]={\bb Z}[s]/\varphi_p(s)
@>>>
{\bb Z}[\xi_{pq}]/\varphi_p(\xi_{pq})={\bb Z}[s]/(\varphi_{pq}(s),\varphi_p(s))
\end{CD}
\]
The compositions to ${\bb Z}[\xi_p]$ and ${\bb Z}[\xi_{pq}]$ are respectively $\lambda$ and $\Lambda$. Therefore the images of an invertible $x=\sum x_it^i \in ({\bb Z}[G]/\Sigma_G)\langle t\rangle^*$ are respectively $\lambda(x_{i(\lambda)})t^{i(\lambda)}\in {\bb Z}[\xi_p]^*\langle t\rangle$ and $\lambda(x_{i(\Lambda)})t^{i(\Lambda)}\in {\bb Z}[\xi_{pq}]^*\langle t\rangle$. Both are further mapped to $({\bb Z}[\xi_{pq}]/\varphi_p(\xi_{pq}))^*\langle t\rangle$. To show $i(\Lambda)=i(\lambda)$, therefore, we only need to show that ${\bb Z}[\xi_{pq}]/\varphi_p(\xi_{pq})\ne 0$. 

By ${\bb Z}[\xi_{pq}]/\varphi_p(\xi_{pq})={\bb Z}[s]/(\varphi_{pq}(s), \varphi_p(s))$, the ring is $0$ if and only if $1=\varphi_{pq}(s)u(s)+\varphi_p(s)v(s)$ for some polynomials $u(s),v(s)\in {\bb Z}[s]$. Taking $s=\xi_p$, we get $1=\varphi_{pq}(\xi_p)v(\xi_p)$. By
\[
\varphi_{pq}(s)
=\frac{(s^{pq}-1)(s-1)}{(s^p-1)(s^q-1)}
=\frac{1+s^p+s^{2p}+\dots+s^{(q-1)p}}{1+s+s^2+\dots+s^{q-1}},
\]
we have 
\[
\varphi_{pq}(\xi_p)
=\frac{q}{1+\xi_p+\xi_p^2+\dots+\xi_p^{q-1}}.
\]
Then $1=\varphi_{pq}(\xi_p)v(\xi_p)$ means $1+\xi_p+\xi_p^2+\dots+\xi_p^{q-1}=qv(\xi_p)$ is a multiple of $q$ in ${\bb Z}[\xi_p]$. Since $p,q$ are coprime, this is not true. 

We conclude that the ring is non-zero. This implies $i(\Lambda)=i(\lambda)$. Switching $p$ and $q$, we also get $i(\Lambda)=i(\lambda')$. This completes the proof of $i(\lambda)=i(\lambda')$.
\end{proof}

\end{document}